\documentclass[12pt, oneside]{amsart}
  
\usepackage[dvips]{geometry}
\usepackage{graphicx,color}
\usepackage{amssymb}

\usepackage[section]{placeins}

\usepackage{todonotes}
\usepackage{enumitem,kantlipsum}

\makeatletter
\@namedef{subjclassname@2020}{%
  \textup{2020} Mathematics Subject Classification}
\makeatother

\usepackage{enumitem,amssymb}
\newlist{todolist}{itemize}{2}
\setlist[todolist]{label=$\square$}

\usepackage{amsfonts}
\usepackage{amsmath}
\usepackage{amsthm}
\usepackage{euscript}
\usepackage{amssymb}
\usepackage{mathrsfs}
\usepackage{xypic}
\usepackage{tikz}
\usepackage{caption}
\usepackage{subfig}
\usepackage{mathtools}

\xyoption{all}

\usepackage[colorlinks]{hyperref}

\theoremstyle{plain}
\newtheorem{thm}{Theorem}[section]

\newtheorem{lem}[thm]{Lemma}
\newtheorem{prop}[thm]{Proposition}

\theoremstyle{definition}

\theoremstyle{remark}
\newtheorem{rem}[thm]{Remark}
\newtheorem{ex}[thm]{Example}
\numberwithin{equation}{section}

\theoremstyle{conjecture}

\newcommand{\Z}{\mathbb Z}

\newcommand{\GG}{\mathbb G}
\newcommand{\PP}{{\mathbb P}}

\newcommand{\OO}{\mathcal{O}}

\newcommand{\II}{\mathcal{I}}

\newcommand{\cA}{\mathcal{A}}

\renewcommand{\ge}{\geqslant}

\renewcommand{\le}{\leqslant}

\DeclareMathOperator{\HH}{H} 
\DeclareMathOperator{\hh}{h}
 
\DeclareMathOperator{\Hom}{Hom}
\DeclareMathOperator{\Pfaff}{Pfaff}

\DeclareMathOperator{\rk}{rk}
\DeclareMathOperator{\Ker}{Ker}

\DeclareMathOperator{\T}{T}

\DeclareMathOperator{\SL}{\mathrm{SL}}

\begin{document}

\sloppy

\title[]{Quadric surfaces in the Pfaffian hypersurface in $\PP^{14}$}    

\author{Ada Boralevi}
\address{Dipartimento di Scienze Matematiche \lq\lq G. L. Lagrange\rq\rq, Politecnico di Torino, Corso Duca degli Abruzzi 24, 10129 Torino, Italy}
\email{\href{mailto:ada.boralevi@polito.it}{ada.boralevi@polito.it}}

\author{Maria Lucia Fania}
\address{Dipartimento di Ingegneria e Scienze dell'Informazione e Matematica, Universit\`a degli Studi dell'Aquila, via Vetoio, Loc. Coppito, 67100 L'Aquila, Italy}
\email{\href{mailto:marialucia.fania@univaq.it}{marialucia.fania@univaq.it}}

\author{Emilia Mezzetti}
\address{Dipartimento di Matematica e Geoscienze, Sezione di Matematica e Informatica,  Universit\`a degli Studi di Trieste, Via Valerio 12/1, 34127 Trieste, Italy}
\curraddr{}
\email{\href{mailto:mezzette@units.it}{mezzette@units.it}}

\thanks{The second and third named author are supported by PRIN 2017SSNZAW. The third named author is also supported by FRA of the University of Trieste. The first named author has been partially supported by MIUR grant Dipartimenti di Eccellenza 2018-2022 (E11G18000350001). All authors are members of INdAM--GNSAGA}



\begin{abstract} 
We study smooth quadric surfaces in the Pfaffian hypersurface in $\PP^{14}$ parameterising $6 \times 6$ skew-symmetric matrices of rank at most 4, not intersecting the Grassmannian $\GG(1,5)$. Such surfaces correspond to  quadratic systems of skew-symmetric matrices of size 6 and constant rank 4, and give rise to a globally generated vector bundle $E$ on the quadric. We analyse these bundles and their geometry, relating them to linear congruences of lines in $\PP^5$.
\end{abstract}
\maketitle

\section{Introduction}



Denote by $V_{n+1}$ a complex vector space of dimension $n+1$. Recall that, after fixing a basis, skew-symmetric tensors in $\wedge^2 V_{n+1}$ can be interpreted as skew-symmetric matrices of size $n+1$, and that the Grassmannian of lines in $\PP^{n}=\PP(V_{n+1})$ corresponds to matrices of rank $2$. A linear congruence in $\PP^n$ is a $(n-1)$-dimensional linear section of $\GG(1,n)$, given itself by the intersection $\GG(1,n) \cap \Delta$, where $\Delta$ is a linear space of codimension $n-1$. The space $\Delta$ is therefore given by the intersection of $n-1$ hyperplanes, that, in turn, correspond to points in the dual space $\check{\PP}(\wedge^2V_{n+1})$, generating a $(n-2)$-space $\check{\Delta}$. 

\smallskip

The study and classification of linear congruences in $\PP^n$ is a classical topic, that has recently found interesting applications in different areas, such as, for instance, systems of conservation laws of Temple  type \cite{agafonov_ferapontov}, degree of irrationality \cite{bastianelli_cortini_depoi,  bastianelli_depoi_ein_lazarsfeld_ullery},  foliations \cite{fassarella}. Thus far, a complete classification is known only for values of $n\leq 4$ \cite{castelnuovo, DePoi_congruences, DePoi_Mezzetti}.

\smallskip

In this article we give a contribution to the study of linear congruences in $\PP^5$, that amounts to describing all special positions 
of the 3-space $\check{\Delta}$ with respect to the dual Grassmannian $\check{\GG}(1,5)$ and to its singular locus.
Let us consider the Grassmannian of lines in $\PP^5=\PP(V_6)$: 
\[\GG(1,5) \hookrightarrow \PP(\wedge^2 V_6) = \PP^{14}.\]
There is a natural filtration, based on the rank of tensors, namely
\[\GG(1,5) \subset \sigma_2(\GG(1,5)) \subset \PP(\wedge^2 V_6) = \PP^{14},\]
corresponding to $6 \times 6$ skew-symmetric matrices of 
\[\{\rk \le 2\} \subset \{\rk \le 4\} \subset \{\rk \le  6\} = \PP(\wedge^2 V_6),\] 
where we denote by $\sigma_2(\GG(1,5))$ the variety of secant lines to the Grassmannian.

\smallskip

Inside the dual space $\check{\PP}^{14}$ there lives the dual variety $\check{\GG}(1,5)$ parameterising hyperplanes tangent to $\GG(1,5)$: it is the cubic hypersurface of $6 \times 6$ skew-symmetric matrices defined by the equation $\Pfaff=0$, so it corresponds to matrices of $\rk \le 4$: $\check{\GG}(1,5) \simeq  \sigma_2(\GG(1,5))$. Its singular locus is naturally isomorphic to $\GG(3, 5)\simeq\GG(1,5)$, and it is formed by hyperplanes tangent to $\GG(1,5)$ at the points corresponding to the lines of a $\PP^3$, and the associated skew-symmetric matrix has rank 2. 

\smallskip

If the 3-space $\Delta$ is general, the intersection $\check{\GG}(1,5) \cap \check{\Delta}$ is a cubic surface $S$. Otherwise $\check{\Delta}\subset\check{\GG}(1,5)$, but in this second case $\check{\Delta}$ meets the rank $2$ locus (see \cite{Manivel_Mezzetti}).

\smallskip

An interesting case arises when the intersection $\check{\GG}(1,5) \cap \check{\Delta}$ is a reducible cubic surface $S$ not intersecting $\GG(1, 5)$, having a plane and a smooth quadric surface as irreducible components. Then such a plane (respectively quadric surface) can be interpreted naturally as a linear (respectively quadratic) system of skew-symmetric matrices of constant rank 4, of projective dimension 2.

\smallskip

Planes of this type have been completely classified in \cite{Manivel_Mezzetti}: up to the action of $PGL_6$ there are exactly four different orbits, all of dimension $26$; they correspond to the double Veronese embeddings of $\PP^2$ in $\GG(1,5)$, or, equivalently, to rank $2$ globally generated vector bundles on $\PP^2$ with first Chern class $c_1=2$ (see \cite{sierra_ugaglia_gg}).
Indeed, given a plane of $6 \times 6$ skew-symmetric matrices of constant rank 4, there is an exact sequence of the form
\[\label{succ_piano}
0 \to E^*(-1) \to \OO_{\PP^2}(-1)^6 \to \OO_{\PP^2}^6 \to E\to 0,
\]
where $E$ is a rank two vector bundle on $\PP^2$ with $c_1(E)=2$. 

\medskip

The aim of this note is to study the embeddings  of smooth quadric surfaces
\[{Q \subset \check{\GG}(1,5) \setminus \GG(1,5).}\] 
The hope of achieving for quadric surfaces the same kind of classification obtained in the case of planes fades immediately, after a quick parameter count shows the existence of an infinite number of orbits.

However, since rank 2 globally generated bundles on a smooth quadric surface are classified (see \cite{Ballico_Huh_Malaspina}), we have considered the following problems: first, understanding which of these vector bundles are associated to a quadratic system of skew-symmetric matrices of constant rank 4; second, studying the geometry of the found examples, relating them to linear congruences of lines in $\PP^5$.

\smallskip

Our main result is the existence Theorem \ref{main theorem}, that gives the complete list of rank $2$ globally generated vector bundles on $Q$ associated to a quadratic system of skew-symmetric matrices of size $6$ and constant rank $4$ (see Section 3 for a more precise statement).

\smallskip

Our techniques rely on the already mentioned classification of planes contained in 
$\check{\GG}(1,5) \setminus \GG(1,5)$, and on a study of the geometry of the bundles involved. More precisely, we construct examples of such quadric surfaces either by considering directly the case of decomposable bundles (Section 4), or by constructing bigger size matrices and then projecting them to desired size ones with a projection technique (Section 5), or else by extending some known examples on $\PP^2$ to a suitable 3-dimensional space (Section 6).

We make frequent use of Macaulay2 software \cite{M2} to study details about the geometry of our examples.

\smallskip

To the best of our knowledge, ours is the first instance of the study of nonlinear spaces contained in these orbits. While some of the ideas and proofs working for the linear case still hold on the quadric surface, 
there are a few differences to note, such as the fact that the rank 2 vector bundles that we construct are globally generated (so they define morphisms to the Grassmannian $\GG(1,5)$), but they do not always define embeddings the way they did in the linear case.

\smallskip

Finally, it is worth mentioning some interesting related work: in the paper \cite{Ferapontov_Manivel}, Ferapontov-Manivel have considered a problem kindred to ours, that admits an interpretation in terms of integrable systems: there, they are interested in 3-dimensional linear spaces $\PP^3 \subset \check{\GG}(1,5)$ satisfying some extra condition. In \cite{Comaschi}, Comaschi studied and classified (stable) $\SL(V_6)$-orbits of linear systems in 
$\PP(\wedge^2 V_6)$, whose generic element is a tensor of rank 4, thus generalising the work of \cite{Manivel_Mezzetti} in a different direction with respect to ours.

\section{Quadrics of skew-symmetric matrices of constant rank and vector bundles}

Let $Q$ be a smooth quadric surface, isomorphic to $\PP^1 \times \PP^1$ and embedded into $\PP^3$ through the Segre map, and let us call $\pi_i$ the projections over $\PP^1$. Any line bundle over $Q$ is of the form $\OO_{Q}(a,b)=\pi_1^*(\OO_{\PP^1}(a)) \otimes \pi_2^*(\OO_{\PP^1}(b))$. For the sake of brevity, we denote $\OO_Q(a,a)$ by $\OO_Q(a)$. Given a vector bundle $E$ over $Q$, we write $E(a,b)$ (respectively $E(a)$) for the tensor product $E \otimes \OO_Q(a,b)$ (resp. $E \otimes \OO_Q(a)$). We also write $c_1(E)=(a,b)$ to mean $c_1(E)=\OO_Q(a,b)$.

\medskip

The existence of a smooth quadric surface $Q \subset \check{\GG}(1,5) \setminus \GG(1,5)$ of skew-symmetric matrices of size $6$ and constant rank $4$ entails a long exact sequence of vector bundles on $Q$, of the form:
\begin{equation}\label{succ_quadrica}
0 \to E^*(-1) \to \OO_Q(-1)^6 \xrightarrow{A} \OO_Q^6 \to E\to 0,
\end{equation}
where $E$ is a rank 2 globally generated vector bundle on the surface, satisfying some non-degeneracy conditions in cohomology. Moreover, the skew-symmetry of the map $A$ above implies a symmetry of sequence \eqref{succ_quadrica}: in particular there is an isomorphism $E\simeq E^*(1)$. We refer to \cite{bo_fa_me} for details. 

\medskip

Similarly to what happens in the linear case, the Chern classes of a bundle $E$ fitting in a long exact sequence of type \eqref{succ_quadrica} must meet some requirements.

\begin{prop}\label{prop_chern} 
Let $E$ be a vector bundle, fitting in an exact sequence of the form \eqref{succ_quadrica} as cokernel of a skew-symmetric matrix of size $6$ and constant rank $4$ over a quadric surface $Q$. Then the Chern classes of $E$ must satisfy $c_1(E)=(2,2)$ and $0 \le c_2(E) \le 6$.
\end{prop}

\begin{proof} We split sequence \eqref{succ_quadrica} into two short exact sequences:
\begin{align}
\label{LHS} &0 \to E^*(-1) \to \OO_Q(-1)^6 \to F\to 0,\\
\label{RHS} &0 \to F \to \OO_Q^6 \to E\to 0,
\end{align}
and compute invariants. From \eqref{RHS} we deduce $c_1(E)=(a,b)=-c_1(F)$, while from \eqref{LHS} we get $c_1(E^*(-1)) +(-a, -b)=(-6,-6)$. Since $\rk(E)=2$, we have an isomorphism $E^* \simeq E(-a,-b)$, and hence $c_1(E^*(-1))=(-a-2,-b-2)$. Putting all together we conclude that $(a,b)=(2,2)$. 

Moving on to the bounds on $c_2$, we first remark that the globally generated bundle $E$ must have $c_2(E) \ge 0$. For the upper bound, we tensor the two sequences \eqref{LHS} and \eqref{RHS} by $\OO_Q(-1,0)$ and compute cohomology.
We get that $\HH^1(E(-1,0))=\HH^2(E(-1,0))=0$, so that $\chi(E(-1,0))=\hh^0(E(-1,0)) \ge 0$. Computing the same Euler characteristic using Chern classes, we get $\chi(E(-1,0))=6-c_2(E) \ge 0$.
\end{proof}

\smallskip

\begin{rem}\label{rem_3fold} 
Consider the rational Gauss map $\gamma:\check{\GG}(1,5) \dashrightarrow \GG(1,5)$, associating to a tangent hyperplane its unique tangency point. It is defined by the partial derivatives of $\Pfaff$, the generic $6\times 6$ Pfaffian determinant, that is, by the $4 \times 4$ principal minors' Pfaffians. Given a quadric surface $Q$ contained in $\check{\GG}(1,5) \setminus \GG(1,5)$ the equations defining $\gamma$ cannot all vanish on $Q$, since the rank there is constant and equal to 4, therefore the restriction
\[\gamma|_Q: Q \to \gamma(Q) \subset \GG(1,5)\]
is a regular map. For a line $\ell \in \GG(1,5)$, the fibre of $\gamma$ over $\ell$ consists of all hyperplanes $H$ tangent to $\GG(1,5)$ such that $\T_\ell\GG(1,5) \subseteq H$, so $\gamma^{-1}(\ell) \simeq (\T_\ell\GG(1,5))\check{\vrule height1.3ex width0pt} \,$ is a $5$-dimensional linear space. 

More in detail, if we choose a basis $\{e_0,\ldots,e_5\}$ of $V_6$ and $\ell=<e_0,e_1>$, then the tangent space $\T_\ell\GG(1,5)$ is defined by equations $p_{ij}=0$ for $i \ge 2$, so it 
coincides with the space of $6 \times 6$ skew-symmetric matrices having all zero entries in the first two rows and columns, that is, the space spanned by a sub-Grassmannian $\GG(1,3)$. Therefore,  the general element of 
$\T_\ell\GG(1,5)$ has rank 4, and those of rank 2 form a quadric hypersurface.
This means that the fibres of $\gamma|_Q$ are of the form $(\T_\ell\GG(1,5))\check{\vrule height1.3ex width0pt} \,\cap Q$, the intersection of a 5-dimensional linear space with a quadric surface. If $(\T_\ell\GG(1,5))\check{\vrule height1.3ex width0pt}\, \cap Q$ had positive dimension, the rank on $Q$ would not be constant, hence these fibres must consist of either 1 or 2 points: in other words, $\deg(\gamma|_Q)$ is either 1 or 2.

From the vector bundle point of view, the Gauss map restricted to $Q$ is given by the rank 2 bundle $E$ from sequence \eqref{succ_quadrica}; indeed, recall from \cite{Fania_Mezzetti} (and refer to the excellent notes \cite{note_grass_arrondo} for details) that, given a globally generated rank 2 vector bundle $E$ on $Q$, and given a fixed $(N+1)$-dimensional  vector subspace $V_{N+1}\subset\HH^0(E)$ generating the global sections, we can construct a map $\varphi_E: Q \to \GG(1,N)$ from $Q$ to the Grassmannian of lines in 
$\PP^N=\PP (V_{N+1})$. 

The map $\varphi_E$ coincides with $\gamma|_Q$ and is therefore a regular map of degree 1 or 2. It is worth noticing that a similar reasoning in \cite[Proposition 2.4]{Fania_Mezzetti} entailed that the map $\varphi_E$ was an embedding. Here, the fact that our system of constant rank matrices is quadratic makes the difference: in what follows we will find examples where $\varphi_E$ is not an embedding.


Consider also $\PP(E)$, the projective bundle associated to $E$. One can prove that $\varphi_E$ is equivalent to a map $\bar{\varphi}_E:  \PP(E) \to \PP^N$ of the corresponding ruled variety obtained by taking the union of all lines defined by the points of $Q$, having the same degree of  ${\varphi}_E$.  Let $Y$ be the image of $\bar{\varphi}_E$; since $Q$ is a surface, $Y$ is a threefold and the following equality holds:
\begin{equation}\label{relazione chern}
c_2(E)=c_1^2(E) -\deg(\bar{\varphi}_E) \cdot \deg (Y)= 8 -\deg(\varphi_E) \cdot \deg (Y).
\end{equation} 
\end{rem}

\section{Globally generated vector bundles with $c_1=(2,2)$ and main result.}

There is a finite list of vector bundles of rank $2$ on a smooth quadric $Q$ that can appear in an exact sequence of the form (\ref{succ_quadrica}). The first ones that come to mind are of course the ones decomposing as direct sum of two line bundles, that we list below.

\begin{prop}\label{decomp}
Let $E$ be a \emph{decomposable} globally generated vector bundle of rank $2$ on a smooth quadric surface $Q$, fitting into an exact sequence of type \eqref{succ_quadrica}. Then
\[E=\OO_Q(a,b) \oplus \OO_Q(2-a,2-b),\] 
with $0\le a,b \le 2$, and the following cases can occur:
\begin{enumerate}
\item[{\em (DEC1)}]  $a=b=0$,   $E=\OO_Q \oplus \OO_Q(2)$, $c_2(E)=0$;
\item[{\em (DEC2)}]  $a=b=1$, $E=\OO_Q(1) \oplus \OO_Q(1)$, $c_2(E)=2$;
\item[{\em (DEC3)}] $a=2$, $b=1$ (and its symmetric),  $E=\OO_Q(2,1) \oplus \OO_Q(0,1)$, $c_2(E)=2$;
\item[{\em (DEC4)}] $a=2$, $b=0$ (and its symmetric),  $E=\OO_Q(2,0) \oplus \OO_Q(0,2)$, $c_2(E)=4$. 
\end{enumerate}
\end{prop}

\smallskip

Indecomposable globally generated vector bundles with low first Chern class on a smooth quadric surface have been classified in the paper \cite{Ballico_Huh_Malaspina}. The authors prove that there exist such indecomposable and globally generated vector bundles of rank $2$ on $Q$ with $c_1=(2,2)$ if and only if $c_2=3,4,5,6,8$. 

In particular, there are no rank 2 globally generated vector bundles on $Q$ satisfying $c_1(E)=(2,2)$ and $c_2(E)=1$.

\medskip

One of the tools used in \cite{Ballico_Huh_Malaspina} is the notion of index: a pair $(p,q) \in \Z^2$ is an \emph{index} for a globally generated vector bundle $E$ on $Q$ if it is a maximal pair such that the twist $E(-p,-q)$ has global sections: $\HH^0(E(-p,-q) )\neq 0$. Here one considers $(p,q) \ge (p',q')$ if and only if $p \ge p'$ and $q\ge q'$. Since the ordering is only partial, a vector bundle can have more than one index. 

If $(p,q)$ is an index of our bundle $E$ with $p+q \ge 3$, then $E$ decomposes as a direct sum $E=\OO_Q(p,q) \oplus \OO_Q(2-p,2-q)$. On the other hand, if $p+q \le 1$ then $E$ is (Mumford-Takemoto) stable, simply because $E$ has rank $2$, hence its stability is equivalent to the vanishing of the three cohomology groups $\HH^0(E(-1))$, $\HH^0(E(-2,0))$, and 
$\HH^0(E(0,-2))$.

\begin{lem}\label{lemma_index}
Let $E$ be a vector bundle appearing in an exact sequence of type \eqref{succ_quadrica} as cokernel of a skew-symmetric constant rank matrix over the quadric surface $Q$, and let $(p,q)$ be an index for $E$. Then $(q,p)$ is an index for $E$; if $c_2(E) \le 5$, then $(p,q) > (0,0)$, and if $c_2(E)=6$, then $p=q=0$. 
\end{lem}

\begin{proof}
The first statement is an immediate consequence of the symmetry of the construction with respect to the two rulings. The second statement follows from the equality $\hh^0(E(-1,0))=6-c_2(E)$ obtained in the proof of Proposition \ref{prop_chern}.
\end{proof}

\medskip

\begin{prop}\label{spezzamento}
Let $E$ be a globally generated vector bundle of rank $2$ on a smooth quadric surface $Q$, fitting into an exact sequence of type \eqref{succ_quadrica}. If $E$ has $(2,0)$ as index, then it decomposes as a direct sum $\OO_Q(0,2)\oplus\OO_Q(2,0)$.
\end{prop}

\begin{proof}
According to \cite{Ballico_Huh_Malaspina}, if $(2,0)$ is an index, then $E$ arises in the following extension:
\begin{equation}\label{ext 1}
0\to\OO_Q(2,0)\xrightarrow{\phi} E\to\OO_Q(0,2)\to 0.
\end{equation}
Now if $E$ fits into sequence \eqref{succ_quadrica}, then by Lemma \ref{lemma_index} $(2,0)$ is an index if and only if $(0,2)$ is also an index, and this, again from \cite{Ballico_Huh_Malaspina}, is equivalent to an extension of type
\begin{equation}\label{ext 2}
0\to\OO_Q(0,2)\to E\xrightarrow{\psi}\OO_Q(2,0)\to 0.
\end{equation}
The composition $\psi \circ \phi \in \Hom(\OO_Q(2,0),\OO_Q(2,0))$ can either be the zero map or a scalar multiple of the identity. If it were zero, then it would induce a non-zero map $\OO_Q(2,0) \to \Ker(\psi)=\OO_Q(0,2)$, which is impossible. Hence it must be a scalar multiple of the identity, meaning that the extension \eqref{ext 1} must split. 
\end{proof}

\smallskip

We are now ready to list all indecomposable bundles that can appear in the long exact sequence  \eqref{succ_quadrica}; the following result, combined with  Proposition \ref{decomp}, gives a complete picture of all possible cases.

\begin{prop}\label{indecomp}
Let $E$ be an {\em indecomposable} globally generated vector bundle of rank $2$ on a smooth quadric surface $Q$, fitting into an exact sequence of type  \eqref{succ_quadrica}. Then one of the following occurs:
\begin{enumerate}
\item[{\em (IND1)}] $E$ has $(1,1)$ as index, $c_2(E)=3$, and there is a short exact sequence of the form \[0\to\OO_Q\to\OO_Q(1)\oplus\OO_Q(1,0)\oplus\OO_Q(0,1)\to E\to 0;\]
the restriction of $E$ on both rulings of $Q$ is $\OO_{\PP^1}(1)\oplus \OO_{\PP^1}(1)$.
\item[{\em (IND2)}] $E$ has $(1,1)$ as index, $c_2(E)=4$, and there is a resolution of type 
\[0 \to \OO_Q(-1) \to \OO_Q^2 \oplus \OO_Q(1) \to E \to 0;\]
in this case the restriction of $E$ to both rulings is $\OO_{\PP^1}(1) \oplus \OO_{\PP^1}(1)$.
\item[{\em (IND3)}] $E$ has indices $(1,0)$ and $(0,1)$ (hence it is a stable bundle), $c_2(E)=4$, and it fits into the short exact sequence (and its symmetric equivalent)
\[0 \to \OO_Q(1,0) \to E \to \mathcal{I}_Z(1,2) \to 0,\]
where $Z$ is a zero-dimensional scheme of degree $2$. $E$ restricts as $\OO_{\PP^1}(1) \oplus \OO_{\PP^1}(1)$ on one ruling, and as $\OO_{\PP^1} \oplus \OO_{\PP^1}(2)$ on the other one.
\item[{\em (IND4)}] $E$ is a stable bundle having indices $(1,0)$ and $(0,1)$, $c_2(E)=5$, fitting in the exact sequence 
\[0\to\OO_Q\to E\to \II_Z(2)\to 0,\] 
where $Z$ is a zero-dimensional scheme of degree $5$.
\item[{\em (IND5)}] $E$ is a stable bundle having index $(0,0)$ and $c_2(E)=6$, and it fits in the exact sequence \[0\to\OO_Q\to E\to \II_Z(2)\to 0,\]
where $Z$ is a zero-dimensional scheme of degree $6$.
\end{enumerate}
\end{prop}
\begin{proof}
Analyzing the classification from \cite{Ballico_Huh_Malaspina} in light of Proposition \ref{prop_chern}, Lemma \ref{lemma_index}, and Proposition \ref{spezzamento}, we are able to rule out a few cases, and are left with the ones listed above.
\end{proof}

A very natural question is whether all globally generated bundles appearing in Propositions \ref{decomp} and \ref{indecomp} are attained with our construction: a positive answer to this question is our main result.

\smallskip

\begin{thm}\label{main theorem}
Let  $X \subset \PP^{14}$  be the cubic Pfaffian hypersurface parameterising $6 \times 6$ skew-symmetric matrices of rank at most 4. For all cases listed in Propositions \ref{decomp} and \ref{indecomp}, there exists a smooth quadric surface $Q \subset X$, not intersecting the Grassmannian $\GG(1,5)$, giving rise to a long exact sequence of type 
\begin{equation}\tag{\ref{succ_quadrica}}
0 \to E^*(-1) \to \OO_Q(-1)^6 \to \OO_Q^6 \to E\to 0,
\end{equation} 
where the vector bundle $E$ is of the desired type.
\end{thm}

\smallskip

We devote the rest of the paper to a constructive proof of Theorem \ref{main theorem}, that we achieve by giving explicit examples of the vector bundle $E$ in all cases. To this end, we use three different techniques: in Section 4 we use decomposable bundles to settle cases {(DEC1)} and {(DEC2)} from Proposition \ref{decomp}. Then in Section 5 we introduce and develop a projection technique, that allows us to construct case {(DEC4)} from Proposition \ref{decomp}, as well as all 5 instances of Proposition \ref{indecomp}. A different technique is needed for the remaining case {(DEC3)} of Proposition \ref{decomp}: this is done in Section 6.

\section{Construction techniques, part 1: some decomposable bundles}
As anticipated, in this section we construct  examples of smooth quadric surfaces contained in $\check{\GG}(1,5) \setminus \GG(1,5)$ that give rise to the decomposable bundles $\OO_Q \oplus \OO_Q(2)$ and $\OO_Q(1) \oplus \OO_Q(1)$, that is, cases {(DEC1)} and {(DEC2)} from Proposition \ref{decomp}.

\begin{ex}\label{DEC 1}
Consider the decomposable vector bundle {(DEC1)} $E= \OO_Q \oplus \OO_Q(2)$ from Proposition \ref{decomp}, having $c_2(E)=0$.
Since $h^0(E)=10$, the image of the map $\varphi_E: Q\to\GG(1,9)$ represents the lines of a cone over $v_2(Q)$. Then, one can project this cone to $\PP^5$ and get a map from $Q$ to $\GG(1,5)$. Therefore, if one has a smooth quadric surface in $\check{\GG}(1,5) \setminus \GG(1,5)$ corresponding to this bundle, by duality it must be contained in the linear span of a sub-Grassmannian $\GG(1,H)$ where $H\subset\PP^5$ is a hyperplane. But $\GG(1,H)$ has codimension $3$ in its linear span $\PP(\wedge^2 V_5)\simeq\PP^9$, so a \emph{general} quadric surface contained in this $\PP^9$ will be disjoint from $\GG(1,H)$.  After a linear change of coordinates, one can assume that the matrix representing a constant rank map $\OO_Q(-1)^6\to\OO_Q^6$ as in (\ref{succ_quadrica}) is a general $6\times 6$ skew-symmetric matrix of linear forms in four variables, suitably restricted to $Q$.

\smallskip

An explicit example is the following:

\begin{equation}\label{O+O(2)}
\left(\begin{array}{ccccc|c} 
\cdot& a& b& c& {d}& \cdot\\
-a& \cdot& a& b& c& \cdot\\
-b& -a& \cdot& {d}& a& \cdot\\
-c& -b& {-d}& \cdot& b& \cdot\\
{-d}& -c& -a& -b& \cdot& \cdot\\
\hline
\cdot& \cdot& \cdot& \cdot& \cdot& \cdot 
\end{array}\right),
\end{equation}
where for the reader's convenience we have adopted the convention to denote a zero in the matrix by a dot.

Matrix \eqref{O+O(2)} has Pfaffian vanishing on the reducible cubic surface union of the plane $\Pi:\{d=0\}$ and the quadric $Q:\{ad-bc=0\}$. As expected, the vector bundle corresponding to the restriction of (\ref{O+O(2)}) to the plane $\Pi$ is $\OO_{\PP^2} \oplus \OO_{\PP^2}(2)$.
\end{ex}

\begin{ex}\label{DEC 2}
From \cite[Proposition 3.5]{Ballico_Huh_Malaspina} we learn that on the quadric surface there is a rank 2 globally generated vector bundle $\cA_P=\pi_P^*(\T{\PP^2}(-1))$, where $\pi_P: Q\to\PP^2$ is the projection of centre a point 
$P\notin Q$. $\cA_P$ has first Chern class $c_1(\cA_P)=(1,1)$ and it has a locally free resolution of type 
\begin{equation}\label{ris AP}
0 \to \OO_Q(-1) \to \OO_Q^3 \to \mathcal{A}_P \to 0,
\end{equation}
where the first map is given by the equations of the point $P$.
By composing this resolution twisted by $(-1)$ with its dual, and remembering that $\mathcal{A}_P^* \simeq \mathcal{A}_P(-1)$, we find the following long exact sequence:
\begin{equation}\label{succ_AP}
\xymatrix@C-1.4ex@R-3ex{0 \to\OO_Q(-2) \ar@{^{(}->}[r]&\OO_Q(-1)^3\ar[rr] \ar[dr]&& \OO_Q^3 \ar@{->>}[r]& \OO_Q(1) \to 0.\\
&&\mathcal{A}_P(-1)\ar[ur]&&}
\end{equation}
The map in the middle is represented by a $3\times 3$ skew-symmetric matrix of linear forms in four variables, which has  constant rank $2$ outside the point $P$, where it becomes the zero matrix. Therefore it has constant rank two on every quadric disjoint from $P$. For instance, if $P=[1:0:0:1]$, a possible matrix is
\begin{equation}\label{blocco_O(1)}
\begin{pmatrix}
\cdot&a-d&b\\
-(a-d)&\cdot&c\\
-b&-c&\cdot
\end{pmatrix}.
\end{equation}

\medskip

Taking the direct sum of two $3\times 3$ blocks of the type described above, we obtain a $6\times 6$ matrix of constant rank 4 on the quadric $Q$, corresponding to the bundle {(DEC2)} $E=\OO_Q(1)\oplus \OO_Q(1)$.

\medskip

There is an interesting difference of behavior depending on whether or not the two points centre of projections coincide.

More in detail, if the two centres of projections are distinct points $P\neq P'$ not on $Q$, we obtain a $6 \times 6$ matrix, which has constant rank $4$ on $\PP^3\setminus\{P,P'\}$, and rank $2$ at the two points. For instance, taking $P=[1:0:0:1]$, $P'=[0:1:1:0]$, we can construct the matrix
\begin{equation}\label{O(1)^2}
\left(\begin{array}{ccc|ccc}
\cdot& a-d& b&\cdot&\cdot&\cdot\\
-a+d& \cdot& c&\cdot&\cdot&\cdot\\
-b& -c& \cdot& \cdot&\cdot&\cdot\\
\hline
\cdot&\cdot&\cdot& \cdot& a& b-c\\
\cdot&\cdot&\cdot& -a& \cdot& d\\
\cdot&\cdot&\cdot& -b+c& -d& \cdot
\end{array}\right).
\end{equation}
 
It is worth noticing that in this example the $\PP^3$ having coordinates 
$a,b,c,d$ is completely contained in $\check{\GG}(1,5)$: indeed, the rank of the matrix \eqref{O(1)^2} is  at most $4$ on all of $\PP^3$. Constant rank 4 is achieved on any quadric that does not contain the two points $P$ and $P'$, such as the smooth quadric $Q:\{ad-bc=0\}$.

With the notation of Remark \ref{rem_3fold}, the threefold $Y$ corresponding to the matrix (\ref{O(1)^2}) turns out to have $\deg (Y)= 6$, as expected. This has been checked with the help of Macaulay2.
$Y$ can be constructed by taking two isomorphic copies of $Q$ in two disjoint $\PP^3$s, then projecting them $2:1$ to two disjoint  
planes, and taking the union of the family of lines joining pairs of points under this correspondence. Its singular locus is formed by the two planes and a line.

\medskip

If instead we use the same point $P$ as centre of projection for both $3 \times 3$ blocks, the rank of the matrix drops to zero at $P$. For example, using the point $P=[1:0:0:1]$ as centre of projection, we obtain the matrix
\begin{equation}\label{O(1)^2 bis}
\left(\begin{array}{ccc|ccc}
\cdot& a-d& b&\cdot&\cdot&\cdot\\
-a+d& \cdot& c&\cdot&\cdot&\cdot\\
-b& -c& \cdot& \cdot&\cdot&\cdot\\
\hline
\cdot&\cdot&\cdot& \cdot& a-d& b\\
\cdot&\cdot&\cdot& -a+d& \cdot& c\\
\cdot&\cdot&\cdot& -b& -c& \cdot
\end{array}\right).\end{equation}

Its generic rank is again $4$, meaning that again the corresponding $\PP^3$ is completely contained in $\check{\GG}(1,5)$,  and drops to $0$ exactly on the point $P$. Hence we can still consider the smooth quadric $Q:\{ad-bc=0\}$. This time though, while the associated bundle is still case {(DEC2)} $E=\OO_Q(1)\oplus \OO_Q(1)$, the induced threefold $Y$ is the smooth cubic $\PP^1\times\PP^2$ and thus the morphism $\varphi_E: Q \to \GG(1,5)$ has degree 2, and is therefore not an embedding. 

\smallskip

As we underlined in Remark \ref{rem_3fold}, this situation never appears when dealing with linear spaces of dimension two, where $\varphi_E$ is always an embedding $\PP^2 \hookrightarrow \GG(1,5)$.
\end{ex}

\section{Construction techniques, part 2: projection}

An efficient method to construct spaces of matrices of constant rank consists in building bigger size matrices of a given rank, and then projecting them to smaller size matrices of the same rank. This technique was introduced in \cite{Fania_Mezzetti} for the case of $\PP^2$, and later used in \cite{bo_me_piani}, but the results hold in more generality. Indeed, they were already extended to linear spaces of matrices of any size in
 \cite[Proposition 5.1]{adp1}; here, we wish to apply these results to the case of quadrics.
In terms of bundles, this method amounts to expressing the desired rank $2$ bundle $E$ as quotient of a bundle of higher rank having the same Chern polynomial. 

\smallskip

Let us denote by $\sigma_r(X)$ the $r$-th secant variety of a projective variety $X$, that is, the closure of the union of $(r-1)$-planes generated by $r$ independent points of $X$. Now, assume that we have a surface $S$ contained in the stratum $\sigma_r(\GG(1,n)) \setminus \sigma_{r-1}(\GG(1,n))$, i.e. $S$ is a surface of skew-symmetric matrices of size $n+1$ and constant rank $2r$. If we project $\PP^n=\PP(V_{n+1})$ to $\PP^{n-1}=\PP(V_n)$ from a point $O$,  this projection induces another projection $\pi_O$ from $\PP(\Lambda^2 V_{n+1})$ to $\PP(\Lambda^2V_n)$, whose centre is the subspace $\Lambda_O \subseteq \GG(1,n)$ representing all lines through the point $O$.

A point $\omega$ in the stratum $\sigma_r(\GG(1,n)) \setminus \sigma_{r-1}(\GG(1,n))$ can be written in the form $[v_1 \wedge w_1 + \cdots + v_r \wedge w_r]$, where the $v_i$s and $w_i$s are $2r$ linearly independent vectors; the corresponding points generate a subspace $L_\omega$ of $\PP^n$ of dimension $2r-1$. The entry locus of $\omega$ is the sub-Grassmannian $\GG(1, L_\omega)$, 
namely a point in 
$\PP(\wedge^2V_{n+1})$ belongs to some $(r -1)$-plane, which is $r$-secant to $\GG(1, n)$ and contains $\omega$, if and only if it belongs to $\GG(1, L_\omega)$.

\begin{prop}\label{prop:Generalization} 
Let $S\subset\sigma_r(\GG(1,n)) \setminus \sigma_{r-1}(\GG(1,n))$ be a surface of skew-symmetric matrices of size $n+1$ and constant rank $2r$, and let $O \in \PP^n$ be a point such that $S \cap \Lambda_O = \emptyset$. Then the matrices of $\pi_O(S)$ have constant rank $2r$ if and only if the point $O$ does not belong to the union of the spaces $L_\omega$, as $\omega$ varies in $S$. As a consequence, $S$ can be projected to $\sigma_r(\GG(1, 2r +1))$ so that its rank remains constant and equal to $2r$.
\end{prop}

\begin{proof} The proofs of {\cite[Proposition 5.8 and Corollary 5.9]{Fania_Mezzetti}} 
go through step by step; we report them for the reader's convenience. If 
$\omega=[v_1 \wedge w_1 + \cdots + v_r \wedge w_r]$ is a point of $S$, then $\pi_O(\omega)=[Av_1 \wedge Aw_1 + \cdots + Av_r \wedge Aw_r]$, where $A$ is a matrix representing the projection $\pi_O$. Its rank is strictly less than $r$ if and only if the vectors $v_i$s and $w_i$s can be chosen so that some summand $Av_i \wedge Aw_i$ vanishes: this means precisely that $O \in L_\omega$. The last statement follows from the fact that $\dim \bigcup_{\omega \in S} L_\omega \le \dim S + 2r -1 =2r+1$.
\end{proof}

As mentioned above, from the point of view of vector bundles, projecting to a smaller size matrix means that the associated bundle $E$ appearing in sequence \eqref{succ_quadrica} is a quotient of a higher rank vector bundle $F$, in the sense that they fit into a short exact sequence of type
\begin{equation}\label{quoziente}
0 \to \OO_Q^{\rk F-2} \to F \to E \to 0.
\end{equation}

\smallskip

A logical way to construct bigger matrices (or higher rank bundles, if one prefers) is using \lq\lq building blocks'', that is, taking the vector bundle $F$ in \eqref{quoziente} to be a direct sum of two globally generated bundles with first Chern class $(1,1)$. In order to apply this method, we need to recall the classification of such bundles on $Q$ of any rank.

\begin{prop}\cite{Ballico_Huh_Malaspina}\label{(1,1)} Let $F$ be a vector bundle on a smooth quadric surface $Q$, with $c_1(F)=(1,1)$. Let $r$ be the rank of $F$, and suppose that $F$ has no trivial summands. Then $F$ is one of the following:
\begin{enumerate}
\item[(i)] $\OO_Q(1)$, $r=1$;
\item[(ii)] $\OO_Q(1,0)\oplus\OO_Q(0,1)$, $r=2$, $c_2=1$;
\item[(iii)] $\T\PP^3(-1)|_Q$, $r=3$, $c_2=2$;
\item[(iv)] $\cA_P=\pi_P^*(\T\PP^2(-1))$, where $\pi_P: Q\to\PP^2$ is the projection of centre $P\notin Q$, $r=2$, $c_2=2$.
\end{enumerate}
\end{prop}

We point out that the rank 2 bundle $\cA_P$ from case {\em(iv)} had already  appeared in Example \ref{DEC 2}. The rank $3$ bundle of case {\em(iii)} is the only non-trivial extension of $\cA$ by $\OO_Q$, as shown in \cite[Proposition 5.4]{Ballico_Huh_Malaspina}. 

\medskip

We now want to study the \lq\lq building blocks'' arising from each of the cases above.
\medskip

The discussion in Example \ref{DEC 2} shows that, given a point $P$ outside the quadric $Q$, $\OO_Q(1)$ fits in a short exact sequence of the form 
\[0 \to \mathcal{A}_P^* \to  \OO_Q^3 \to \OO_Q(1) \to 0,\]
which is just \eqref{ris AP} dualised. Therefore the building block corresponding to case {\em (i)}
is a $3\times 3$ matrix of the form 
\begin{equation}\label{blocco_O(1)_2}
\begin{pmatrix}
\cdot&\ell_1&\ell_2\\
-\ell_1&\cdot&\ell_3\\
-\ell_2&-\ell_3&\cdot
\end{pmatrix},
\end{equation} where $\ell_1, \ell_2, \ell_3$ are linear forms in four variables such that the equations $\ell_1=\ell_2=\ell_3=0$ define the point $P$.

To see what kind of building block corresponds to case {\em(ii)}, we remark  that the decomposable bundle  $\OO_Q(1,0) \oplus \OO_Q(0,1)$  gives rise to an exact sequence of the form
\[0 \to \OO_Q(-2,-1)\oplus\OO_Q(-1,-2) \to \OO_Q(-1)^4 \to \OO_Q^4 \to \OO_Q(1,0)\oplus\OO_Q(0,1)\to 
0.\]
A corresponding building block is, for instance, the
$4\times 4$ skew-symmetric matrix: 
\begin{equation}\label{blocco_(ii)}
\left(\begin{array}{cccc}
\cdot&\cdot&a&b\\
\cdot&\cdot&c&d\\
-a&-c&\cdot&\cdot\\
-b&-d&\cdot&\cdot
\end{array}\right).
\end{equation}
It can be interpreted as a quadric surface contained in $\check{\GG}(1,3)$, and more precisely it is a linear section of $\check{\GG}(1,3)$ cut out by two hyperplanes. It represents a linear congruence of lines in $\PP^3$, formed by the lines meeting two fixed skew lines in $\PP^3$ (see for instance \cite{DePoi_congruences}).

The bundle appearing in case {\em (iii)} gives rise to 
the $5\times 5$ skew-symmetric block
\begin{equation}\label{blocco_(iii)}
\left(\begin{array}{c|cccc}
\cdot&a&b&c&d\\
\hline
-a&\cdot&\cdot&\cdot&\cdot\\
-b&\cdot&\cdot&\cdot&\cdot\\
-c&\cdot&\cdot&\cdot&\cdot\\
-d&\cdot&\cdot&\cdot&\cdot
\end{array}\right).
\end{equation}
The map represented by matrix  \eqref{blocco_(iii)} is obtained by composing the Euler sequence on $\PP^3$ restricted to $Q$:
\[0 \to \OO_Q(-1) \to \OO_Q^4 \to {\T\PP^3}(-1)|_Q \to 0,\]
with its dualised sequence, so that we get a long exact sequence of the form
\[\xymatrix@C-1.6ex@R-3ex{({\T\PP^3}(-1)|_Q)^*(-1) \ar@{^{(}->}[r]&\OO_Q(-1)^5\ar[rr] \ar[dr]&& \OO_Q^5 \ar@{->>}[r]& {\T\PP^3}(-1)|_Q.\\
&&\OO_Q \oplus \OO_Q(-1)\ar[ur]&&}\]

Finally, since the rank two bundle $\cA_P$ is a quotient of ${\T\PP^3}(-1)|_Q$, a corresponding building block can be obtained by projection from (\ref{blocco_(iii)}). For example, if we project from the point $P=[1:0:0:1]$, we can take
\begin{equation}\label{blocco_(iv)}
\left(\begin{array}{cccc}
\cdot&b&c&d-a\\
-b&\cdot&\cdot&-b\\
-c&\cdot&\cdot&-c\\
a-d&b&c&\cdot
\end{array}\right).
\end{equation}


\bigskip

Let us now give more details on how we apply the projection technique: first, we consider a direct sum of two of the bundles with $c_1=(1,1)$ appearing in Proposition \ref{(1,1)}, and the direct sum of two corresponding matrices. Then, we take a quotient of rank two of this bundle and the corresponding projection of the matrix. We compute the Chern class $c_2$ of the quotient, and we try to identify the rank two bundle so obtained. 

The possible values of $c_2$ that one can obtain are the following:
\begin{enumerate}
\item{$\OO_Q(1)\oplus\OO_Q(1)$}: $c_2=2$, the rank is $2$, there is no projection;
\item{$\OO_Q(1)\oplus\OO_Q(1,0)\oplus\OO_Q(0,1)$}: $c_2=3$, $E$ is of type (IND1);
\item{$\OO_Q(1)\oplus \T{\PP^3}(-1)|_Q$}: $c_2=4$;
\item{$\OO_Q(1,0)^{2}\oplus\OO_Q(0,1)^{2}$}: $c_2=4$;
\item{$\OO_Q(1,0)\oplus\OO_Q(0,1)\oplus\T{\PP^3}(-1)|_Q$}: $c_2=5$, $E$ is of type (IND4);
\item{$\T{\PP^3}(-1)|_Q^{2}$}: $c_2=6$, $E$ is of type (IND5).
\end{enumerate}

\medskip

In the three instances (2), (5), (6), corresponding to second Chern class $3,5,6$ respectively, there is only one possible 
globally generated bundle having these invariants, namely the ones appearing in cases (IND1), (IND4), (IND5) from Propositions \ref{decomp} and \ref{indecomp}. We start by giving explicit 
examples for all these three cases. 

%

\begin{ex}\label{IND1}
The quotient of type 
\[0 \to \OO_Q \to  \OO_Q(1)\oplus\OO_Q(1,0)\oplus\OO_Q(0,1)\to E \to 0 \] 
has $c_2(E)=3$, therefore the vector bundle $E$ corresponds to case {(IND1)} in Proposition \ref{indecomp}. A constant rank matrix obtained via the projection technique is:
\begin{equation}\label{c2=3}
\left(\begin{array}{cccccc}
\cdot& b+c& -a+d& -a+d& \cdot& -a+d\\
-b-c& \cdot& -b& -b& 0& -b\\
a-d& b&\cdot& \cdot& a& -b\\
a-d& b& \cdot& \cdot& -c& d\\
\cdot& \cdot& -a&c& \cdot& \cdot\\
a-d& b& b& -d& \cdot& \cdot
\end{array}\right).
\end{equation}

Its Pfaffian defines the cubic surface union of the plane $\Pi:\{b+c=0\}$ and the quadric surface $Q:\{ad-bc=0\}$. With the help of Macaulay2, we get that $Y$ is a threefold of degree $5$ as expected, and that its singular locus is the union of the line $x_2=x_3=x_4=x_5=0$ and the two points $[0:0:1:-1:0:0]$ and $[0:0:0:0:1:0]$.
\end{ex}

\begin{ex}\label{IND 4}
The quotient of type 
\[0 \to \OO_Q^3 \to  \OO_Q(1,0) \oplus \OO_Q(0,1)\oplus \T\PP^3(-1)|_Q \to E \to 0 \]
has $c_2(E)=5$, therefore the bundle $E$ corresponds to case {(IND4)} in Proposition \ref{indecomp}. 
A constant rank matrix obtained via the projection technique is:
\begin{equation}\label{c2=5}
\left(\begin{array}{cccccc}
\cdot& -b& b& \cdot& \cdot& -a\\
b& \cdot& \cdot& -d& \cdot& (b-c)\\
-b& \cdot& \cdot& a& c& -d\\
\cdot& d& -a& \cdot& \cdot& c\\
\cdot& \cdot& -c& \cdot& \cdot& \cdot\\
a& -(b-c)& d& -c& \cdot& \cdot
\end{array}\right).
\end{equation}

Its Pfaffian defines the cubic surface union of the plane $\Pi:\{c=0\}$ and the quadric surface $Q:\{ad-bc=0\}$. We find that $Y$ is a threefold of degree $3$ as expected, and it is singular 
at four points.
\end{ex}

\begin{ex}\label{IND 5}
A quotient of type
\begin{equation}\label{copie del tangente}
0 \to \OO_Q^4 \to (\T\PP^3(-1)|_Q)^2 \to E \to 0
\end{equation}
has $c_2(E)=6$, therefore the bundle $E$ corresponds to case {(IND5)} in Proposition \ref{indecomp}.  
A constant rank matrix obtained via the projection technique is the following:
\begin{equation}\label{c2=6}
\left(\begin{array}{cc|cccc}
\cdot& a& b& c& d& \cdot\\
-a& \cdot& a& b& c& d\\
\hline
-b& -a& \cdot& \cdot& \cdot& \cdot\\
-c& -b& \cdot&\cdot& \cdot& \cdot\\
-d& -c& \cdot& \cdot& \cdot& \cdot\\
\cdot& -d& \cdot& \cdot& \cdot& \cdot
\end{array}\right).
\end{equation}

Matrix \eqref{c2=6} has generic rank $4$, meaning that in this example the $\PP^3$ having coordinates $a,b,c,d$ is completely contained in $\check{\GG}(1,5)$: 
the rank drops to $2$ exactly on the point $P=[1:0:0:0]$ 
so we can work on the quadric $Q:\{ad-bc=0\}$.
The induced threefold is $Y=\bar{\varphi}_E(\PP(E))=\PP^3$, therefore this is another instance where $\varphi_E$ is not an embedding, and has precisely degree 2.

\begin{rem}
We note that in the $10\times 10$ matrix, direct sum of two blocks of type (\ref{blocco_(iii)}),  all the non-zero elements are contained in two rows and columns. This means that  the $\PP^3$ represented by this matrix is entirely contained in the tangent space to the Grassmannian $\GG(1,10)$ at a point $\ell$.  After projecting and restricting to the quadric, we see that $Q$ is contained in the tangent space to $\GG(1,5)$ at the point $\ell'$ projection of $\ell$. Therefore, when we apply the map $\varphi_E=\gamma|_Q$ to $Q$, the image is contained in $\GG(1,H)$, where $H$ is the $\PP^3$ dual of $\ell'$. Hence $\bar{\varphi}_E(\PP(E))$ is contained in a $\PP^3$. It follows that $\varphi_E$  has  degree $2$ for any choice of projection.
\end{rem}

\begin{rem}
An interesting observation is that the vector bundle $E$ from \eqref{copie del tangente}, quotient of a direct sum of copies of $\T\PP^3(-1)|_Q$, attains the maximal possible value of the second Chern class $c_2(E)$, from Proposition \ref{prop_chern}. 
This can be seen as a ``quadratic counterpart'' to \cite[Proposition 3.2]{bo_me_piani}: there, in the classification of dimension 2 linear spaces of matrices, an upper bound for the second Chern class was found. The bundles whose $c_2$ attained the maximal values were precisely the ones obtained on $\PP^2$ as quotients of a direct sum of copies of $\T\PP^2(-1)$.
\end{rem}
\end{ex}

\smallskip

So far we have used the projection technique to construct examples where the value of the second Chern class was associated with a unique vector bundle in Propositions \ref{decomp} and \ref{indecomp}. We now move on to the trickier case $c_2=4$: we will see that, depending on the choice of the centre of projection, we can obtain all three corresponding cases, namely {(DEC4)}, {(IND2)}, and {(IND3)}.

\begin{ex}\label{c_2=4}
Taking the $8 \times 8$ skew-symmetric matrix direct sum of two blocks of type \eqref{blocco_(ii)}, having constant rank $4$ on the quadric surface $Q:\{ac-bd=0\}$, we obtain a quotient of type
\begin{equation}\label{es_quoz_1}
0 \to \OO_Q^2 \to \OO_Q(1,0)^2 \oplus \OO_Q(0,1)^2 \to E \to 0.
\end{equation}

After computer tests with Macaulay2, we ended up with the following three cases to be considered.

\medskip

Projecting from the line $L$  
of equations $x_0-x_2=x_0+x_1-x_3=2x_2-x_3+x_4=x_3-x_4-x_5=2x_4+x_5-2x_6=x_5-2x_7=0$, we obtain the following $6 \times 6$ skew-symmetric matrix whose rank is constant and equal to 4 on $Q$:

 \begin{equation}\label{c2=4 con proiezione2}
\scriptsize {\left(\begin{array}{ccc|ccc}
\cdot& a-b+c-d&2a-b+2c-d& a+c& b+d& \cdot\\ 
-a+b-c+d& \cdot&2c-d& c& d& \cdot\\
-2a+b-2c+d& -2c+d& \cdot& -a&\cdot& -a+b\\
\hline
-a-c&-c& a& \cdot&-a-c& b\\
-b-d& -d& \cdot&a+c& \cdot& a-b+c-d\\
\cdot& \cdot& a-b& -b& -a+b-c+d& \cdot\\
\end{array}\right).}
\end{equation}

Its Pfaffian vanishes on the quadric $Q$ and on the plane $\Pi:\{a+b=0\}$. 
The threefold $Y$ from Remark \ref{rem_3fold} has degree $4$, which is consistent with the fact that the associated bundle $E$ in \eqref{es_quoz_1} has $c_2(E)=4$. Its singular locus consists of four points.

One can see that the restriction of $E$ to both rulings of the quadric is $\OO_{\PP^1}(1) \oplus \OO_{\PP^1}(1)$, hence $E$ is an indecomposable bundle corresponding to case {(IND2)} in Proposition \ref{indecomp}.

\medskip

 Projecting from the line 
of equations $x_0=x_1-x_6=x_2+x_7=x_3-x_6=x_4-x_6=x_5=0$
 we obtain the following $6 \times 6$ skew-symmetric matrix whose rank is constant and equal to 4 on $Q$:
\begin{equation}\label{c2=4 con proiezione3}
\left(\begin{array}{ccc|ccc}
\cdot& \cdot&a&b&\cdot& \cdot\\ 
\cdot& \cdot&c& d& a& c\\
-a& -c& \cdot&\cdot& -b& -d\\
\hline
-b&-d&\cdot& \cdot&a& c\\
\cdot& -a& b&-a& \cdot&c\\
\cdot&-c& d& -c& -c& \cdot\\
\end{array}\right).
\end{equation}

Its Pfaffian vanishes on the quadric $Q$ and on the plane $\{b-c=0\}$.  
The threefold $Y$ has degree $4$, as expected. Its singular locus is the union of a conic and two points.
The restriction of $E$ to one of the rulings of the quadric is $\OO_{\PP^1}(1) \oplus \OO_{\PP^1}(1)$ and to the other ruling is $\OO_{\PP^1}\oplus \OO_{\PP^1}(2)$, therefore we are dealing with case {(IND3)} from Proposition \ref{indecomp}.

\medskip

Finally, projecting from the line of equation $x_0=x_2=x_3+x_7=x_4+x_1=x_5=x_6=0$ we obtain the following $6 \times 6$ skew-symmetric matrix whose rank is constant and equal to 4 on $Q$:
\begin{equation}\label{O(2,0)+O(0,2)}
\left(\begin{array}{ccc|ccc}
\cdot& a& b& \cdot& \cdot& \cdot\\ 
-a& \cdot& \cdot& -(b+c)& -d& \cdot\\
-b& \cdot& \cdot& -d&\cdot& \cdot\\
\hline
\cdot& b+c& d& \cdot& \cdot& a\\
\cdot& d& \cdot& \cdot& \cdot& c\\
\cdot& \cdot& \cdot& -a& -c& \cdot
\end{array}\right).
\end{equation}

Its Pfaffian vanishes on the quadric $Q$ and on the plane $\{b+c=0\}$. This time again the threefold $Y$  turns out to have degree $4$ as expected. The singular locus of $Y$ consists of two disjoint conics.

Since the restriction of $E$ to both rulings of the quadric is $\OO_{\PP^1} \oplus \OO_{\PP^1}(2)$, this  means that the matrix \eqref{O(2,0)+O(0,2)}'s cokernel is the decomposable bundle $\OO_Q(2,0) \oplus \OO_Q(0,2)$, and that we have constructed an example of case {(DEC4)} from Proposition \ref{decomp}.
\end{ex}

The difference among the three cases in Example \ref{c_2=4} can be explained looking at the position of the line $L\subset \check{\PP^7}$, centre of projection, with respect to four $5$-spaces we now introduce. The vector bundle $F:=\OO_Q(1,0)^2 \oplus \OO_Q(0,1)^2$ defines a map $\psi: Q\to\GG(3,7)$ that can be interpreted as follows. Each direct summand defines a map $\pi_i: Q\to \PP^1$; we fix $4$ general lines $\ell_i$, $i=1\ldots, 4$, in $\PP^7$ and identify them with the codomains of the maps $\pi_i$. Then $\psi$ sends a point $P\in Q$ to the $\PP^3$ generated by the images $\pi_i(P)$.  The duals of the lines $\ell_i$ are the $5$-spaces under consideration. The general case {(IND2)} is obtained when $L$ is disjoint from all the $5$-spaces, in the second case {(IND3)} $L$ meets two of the $5$-spaces, and in the decomposable case {(DEC4)} $L$ meets all of them.

\section{Construction techniques, part 3: extend \& restrict}\label{extres}

To conclude the proof of Theorem \ref{main theorem} there is only one case left, namely the decomposable bundle $\OO_Q(2,1) \oplus \OO_Q(0,1)$, case {(DEC3)} from Proposition \ref{decomp}. It cannot be constructed with the techniques from the previous sections, and we need a different approach.

Recall that we are looking at the smooth quadric surface $Q$ as a quadratic system of skew-symmetric matrices of constant rank 4, of projective dimension 2. The spanned $\PP^3=<Q>$ cannot be entirely contained in $\check{\GG}(1,5) \setminus \GG(1,5)$, therefore two possibilities can occur. The first one is that $\PP^3 \subset \check{\GG}(1,5)$ and $\PP^3 \cap \GG(1,5) \neq \emptyset$: then the general plane $\PP^2 \subset \PP^3$ will be a plane of constant rank matrices, and thus equivalent to one of the four types described in \cite{Manivel_Mezzetti}. The other instance that can arise is that $\PP^3 \nsubseteq \check{\GG}(1,5)$: then the intersection $\PP^3\cap \check{\GG}(1,5)$ will be a cubic surface, union of a quadric $Q$ and a plane of constant rank matrices, again equivalent to one of the types in \cite{Manivel_Mezzetti}.

Thus, if one considers a plane in one of the 4 orbits of \cite{Manivel_Mezzetti}, extends  the associated $6 \times 6$ skew-symmetric matrix to a $\PP^3$, and then restricts it to a quadric surface $Q \subset \PP^3$ that does not intersect the Grassmannian $\GG(1,5)$, one obtains exactly a quadratic system of skew-symmetric matrices of constant rank 4. This should clarify why we call this technique ``extend \& restrict''.

\begin{ex}\label{DEC 3}
We extend a plane of type $\Pi_t$ from \cite[Example 3]{Manivel_Mezzetti} to a $\PP^3$, and then intersect this $\PP^3$ with the Pfaffian hypersurface: the intersection is a cubic surface, union of $\Pi_t$ and a smooth quadric. The corresponding vector bundle on the plane is a Steiner bundle $\mathcal{E}$ on $\PP^2$ fitting in a short exact sequence of type 
\[0 \to \OO_{\PP^2}(-1)^2 \to \OO_{\PP^2}^4 \to \mathcal{E} \to 0.\]

Implementing this idea with the help of Macaulay2, we obtain the following example:
\begin{equation}\label{O(2,1)+O(0,1)}
\left(\begin{array}{cc|cccc} 
\cdot&\cdot & b& c& d& {a}\\
\cdot&\cdot & {a}& b& c& d\\
\hline
-b& {-a}& \cdot& \cdot& {a}& {-a}\\
-c& -b& \cdot& \cdot& \cdot& \cdot\\
-d& -c& {-a}& \cdot& \cdot& \cdot\\
{-a}& -d& {a}& \cdot& \cdot& \cdot
\end{array}\right),
\end{equation}
whose Pfaffian vanishes, as expected, on the cubic surface in $\PP^3$ union of the plane $\Pi:\{a=0\}$ and the quadric $Q:\{ab-c^2+bd-cd=0\}$. 

The resulting threefold $Y$ from Remark \ref{rem_3fold} had degree 6; therefore from equation \eqref{relazione chern} we learn that $\deg (\bar{\varphi}_E)=1$ and $c_2(E)=2$, and hence $E$ splits as the direct sum of two line bundles. More in detail, $Y$ is the union of cones having vertices on a given line; its singular locus is the union of the line itself together with a twisted cubic.

Furthermore, the splitting type of $E$ on the two rulings of the quadric is $\OO_{\PP^1} \oplus \OO_{\PP^1}(2)$ on the first ruling and $\OO_{\PP^1}(1) \oplus \OO_{\PP^1}(1)$ on the second: we conclude that $E$ is the vector bundle $\OO_Q(2,1) \oplus \OO_Q(0,1)$ (or its symmetric equivalent $\OO_Q(1,2) \oplus \OO_Q(1,0)$), that is, we have constructed an example corresponding to case (DEC3). 
\end{ex}

The proof of the main Theorem \ref{main theorem} is now completed.

\begin{rem} 

Concerning this new method a natural question arises: 
since there are exactly four different orbits (up to the action of 
$PGL_6$) of $\PP^2 \subset \check{\GG}(1,5) \setminus \GG(1,5)$ 
described in \cite{Manivel_Mezzetti}, and since we have just seen in 
Example \ref{DEC 3}  that the plane of type $\Pi_t$ does  ``extend \& 
restrict'', one would like to know if this holds true for all the 
planes in the four different orbits.   A positive answer to this last 
question concludes  our paper.

\medskip
Of course, a plane of type $\Pi_5$ from \cite[Example 1]{Manivel_Mezzetti}, that is, a plane contained in 
$\PP(\wedge^2 V_5) \subset \PP(\wedge^2 V_6)$, associated to the split bundle $\OO_{\PP^2} \oplus \OO_{\PP^2}(2)$, will extend \& restrict to the decomposable bundle $\OO_Q \oplus \OO_Q(2)$, case {(DEC1)} from Proposition \ref{decomp}. Matrix \eqref{O+O(2)} is an explicit example.

\medskip

A plane of type $\Pi_p$ from \cite[Example 4]{Manivel_Mezzetti}, corresponding to the Null Correlation bundle on $\PP^3$ restricted to a hyperplane, extends \& restrict to an indecomposable bundle of type {(IND1)} in Proposition \ref{indecomp}. An explicit example is the matrix:
\begin{equation}\label{fibrato c2=3}
\left(\begin{array}{ccc|ccc} 
\cdot& \cdot& {d}& a& b& c\\
\cdot& \cdot& a& c& {d}& \cdot\\
{-d}& -a& \cdot& b& \cdot& {d}\\
\hline
-a& -c& -b& \cdot& \cdot& \cdot\\
-b& {-d}& \cdot& \cdot& \cdot& \cdot\\
-c& \cdot& {-d}& \cdot& \cdot& \cdot
\end{array}\right).
\end{equation}

\medskip

Finally, a plane of type $\Pi_g$ from \cite[Example 2]{Manivel_Mezzetti}, whose corresponding vector bundle is the decomposable bundle $\OO_{\PP^2}(1)\oplus \OO_{\PP^2}(1)$, extends \& restrict to the decomposable bundle 
$\OO_Q(2,0)\oplus\OO_Q(0,2)$, case {(DEC4)} from Proposition \ref{decomp}. An explicit example is the matrix:
\begin{equation}\label{O(2,0)+O(0,2) bis}
\left(\begin{array}{ccc|ccc} 
\cdot& a& -b& {d}&\cdot &\cdot \\
-a& \cdot& c&\cdot & {-d}& \cdot\\
b& -c& \cdot&\cdot &\cdot & {-d}\\
\hline
{-d}& \cdot&\cdot & \cdot& a& b\\
\cdot& {d}&\cdot & -a& \cdot& c\\
\cdot& \cdot& {d}& -b& -c& \cdot
\end{array}\right).
\end{equation}
%
\end{rem}

\bibliographystyle{amsalpha}
\bibliography{biblio_quadriche}

\end{document}